\documentclass[10pt]{article}
\usepackage{amsmath,amssymb}
\usepackage{enumerate}
\usepackage[dvipsnames]{xcolor}

\allowdisplaybreaks[1]

\newtheorem{theorem}{Theorem}[section]

\newtheorem{lemma}[theorem]{Lemma}
\newtheorem{proposition}[theorem]{Proposition}
\newtheorem{corollary}[theorem]{Corollary}

\newtheorem{question}[theorem]{Question}

\newtheorem{conjecture}[theorem]{Conjecture}

\newcommand{\Z}{\mathbb{Z}}

\renewcommand{\ker}{\operatorname{Ker}}
\newcommand{\id}{\operatorname{id}}

\newcommand{\Sym}{\operatorname{Sym}}
\newcommand{\aut}{\operatorname{Aut}}

\newcommand{\soc}{\operatorname{Soc}}

\newcommand{\Aut}{\operatorname{Aut}}

\newcommand{\gr}{\operatorname{gr}}

\newcommand{\Ze}{\operatorname{Z}}

\newenvironment{proof}{\par\noindent{ Proof.}}{$\qed$\par\bigskip}
\newcommand{\qed}{\enspace\vrule  height6pt  width4pt  depth2pt}
\usepackage{color}

\begin{document}
\title{New classes of  IYB groups
}
\author{F. Ced\'o \and J. Okni\'{n}ski
}

\date{}

\maketitle

\vspace{30pt}
 \noindent \begin{tabular}{llllllll}
  F. Ced\'o && J. Okni\'{n}ski \\
 Departament de Matem\`atiques &&  Institute of
Mathematics \\
Universitat Aut\`onoma de Barcelona &&    University of Warsaw\\
08193 Bellaterra (Barcelona), Spain    &&  Banacha 2, 02-097 Warsaw, Poland \\
Ferran.Cedo@uab.cat && okninski@mimuw.edu.pl
\end{tabular}\\

\vspace{30pt} \noindent Keywords:  brace, Yang--Baxter equation, 
solvable group, Sylow tower property, nilpotent groups\\


\begin{abstract}
	It is proven that every finite group of odd order with all Sylow subgroups of nilpotency class at most two is an involutive Yang-Baxter group (IYB group for short), i.e. it admits a structure of left brace. It is also proven that every finite solvable group of even order with all Sylow subgroups of nilpotency class at most two and abelian Sylow $2$-subgroups is an IYB group. These results contribute to the open problem asking which finite solvable groups are IYB, in particular they generalize a result of Ben David and Ginosar \cite{BDG} concerned with finite solvable groups with abelian Sylow subgroups. With the same techniques it is proven that every finite solvable group with all Sylow  subgroups nilpotent of class at most two is isomorphic to the multiplicative group of a skew left brace of nilpotent type. It is also proven that every finite group with the Sylow tower property is isomorphic to the multiplicative group of a skew left brace of nilpotent type.
\end{abstract}

\newpage

\section{Introduction and preliminaries}
The Yang-Baxter equation is an important equation in mathematical physics and it lies at the foundations of quantum groups.
In \cite{drinfeld} Drinfeld suggested to study the set-theoretic solutions of the Yang--Baxter equation. Gateva-Ivanova and Van den Bergh \cite{GIVdB} and Etingof, Schedler and Soloviev \cite{ESS} introduced the class of so called involutive non-degenerate set-theoretic solutions of the Yang--Baxter equation and studied these solutions introducing a special class of groups. Rump in \cite{R07} introduced a new algebraic structure, called brace, to study this class of solutions. Guarnieri and Vendramin in \cite{GV} introduced skew braces to study arbitrary non-degenerate set-theoretic solutions of the Yang--Baxter equation. Because of deep connections to several areas of mathematics, this area has attracted a lot of attention during the last twenty years, see for example \cite{Csurvey} and the references therein. 

Let $X$ be a non-empty set and  let  $r:X\times X \longrightarrow
X\times X$ be a map. For $x,y\in X$ we put $r(x,y) =(\sigma_x (y),
\gamma_y (x))$. Recall that $(X,r)$ is an involutive,
non-degenerate set-theoretic solution of the Yang--Baxter equation
if $r^2=\id$, all the maps $\sigma_x$ and $\gamma_y$ are bijective
maps from  the set $X$ to itself and
$$r_{12} r_{23} r_{12} =r_{23} r_{12} r_{23},$$
where $r_{12}=r\times \id_X$ and $r_{23}=\id_X\times \ r$ are maps
from $X^3$ to itself. Because $r^{2}=\id$, one easily verifies
that $\gamma_y(x)=\sigma^{-1}_{\sigma_x(y)}(x)$, for all $x,y\in
X$ (see for example \cite[Proposition~1.6]{ESS}).

\bigskip
\noindent {\bf Convention.} Throughout the paper a solution of the
YBE will mean an involutive, non-degenerate, set-theoretic
solution of the Yang--Baxter equation. \\

The permutation group of a solution $(X,r)$ of the YBE, \cite{GIC}, is
$$\mathcal{G}(X,r)=\gr(\sigma_x\mid x\in X).$$ 
Etingof, Schedler and Soloviev in \cite{ESS} proved that for every solution $(X,r)$ of the YBE there exists an abelian group $A$, an action of $\mathcal{G}(X,r)$ on $A$ and a bijective 1-cocycle $\mathcal{G}(X,R)\rightarrow A$. Using this, they proved in the same paper that if $(X,r)$ is a finite solution of the YBE, then the group $\mathcal{G}(X,r)$ is solvable. 

A left brace, \cite{R07,CJO}, is a set $B$ with two binary operations, $+$ and
$\circ$, such that $(B,+)$ is an abelian group (the additive group
of $B$), $(B,\circ)$ is a group (the multiplicative group of $B$),
and for every $a,b,c\in B$,
\begin{eqnarray} \label{braceeq}
	a\circ (b+c)+a&=&a\circ b+a\circ c.
\end{eqnarray}
In any left brace $B$  the neutral elements $0,1$ for the
operations $+$ and $\circ$ coincide. Moreover, there is an action
$\lambda\colon (B,\circ)\longrightarrow \aut(B,+)$, called the
lambda map of $B$, defined by $\lambda(a)=\lambda_a$ and
$\lambda_{a}(b)=-a+a\circ b$, for $a,b\in B$. We shall write
$a\circ b=ab$ and $a^{-1}$ will denote the inverse of $a$ for the
operation $\circ$, for all $a,b\in B$. A trivial brace is a left
brace $B$ such that $ab=a+b$, for all $a,b\in B$, i.e. all
$\lambda_a=\id$. The socle of a left brace $B$ is
$$\soc(B)=\{ a\in B\mid ab=a+b, \mbox{ for all
}b\in B \}.$$ Note that $\soc(B)=\ker(\lambda)$, and thus it is a
normal subgroup of the multiplicative group of $B$.

An involutive Yang-Baxter group (IYB group for short) is a finite group isomorphic to the permutation group $\mathcal{G}(X,r)$ of a finite solution of the YBE, \cite{CJR}. Hence by the above result of Etingof, Schedler and Soloviev, every IYB group is solvable. This is equivalent to saying that the multiplicative group of a finite left brace is solvable \cite[Theorem 2.1]{CJR}. In \cite{CJR} it is proven that the class of IYB group includes the following: finite abelian-by-cyclic groups, finite nilpotent groups of class 2 (see also \cite{AW}), direct products and wreath products of IYB groups, semidirect products $A\rtimes H$ with $A$ a finite abelian group and $H$ an IYB group, Hall subgroups of IYB groups, Sylow subgroups of the symmetric groups $\Sym_n$. As a consequence it is also proven that every finite solvable group is isomorphic to a subgroup of an IYB group, and every finite nilpotent group is isomorphic to a subgroup of a nilpotent IYB group.
Finite solvable groups with abelian Sylow subgroups (so called $A$-groups) are IYB groups, \cite[Corollary 4.3]{BDG} and \cite[Theorem 2.1]{CJO20}.
 
On the other hand, Bachiller in \cite{B} shows an example of a finite $p$-group $G$, for a prime $p$, such that $G$ is not an IYB group.  Thus, the following is a natural question.
 
\begin{question}\label{general}
	Let $G$ be a finite solvable group. Suppose that all Sylow subgroups of $G$ are IYB groups. Is $G$ an IYB group? 
\end{question}	
Actually, if all Sylow subgroups of $G$ coming from a Sylow system on $G$ are left braces, then Theorem 3.7 in \cite{rump} gives certain conditions under which these brace structures extend to a structure of left brace on $G$; these conditions are formulated in the equivalent language of the so called affine structures.

	Note that if the answer to this question is affirmative, then the additive Sylow subgroups $P_1,\dots ,P_m$ of the left brace $G$ form a Sylow system  of its multiplicative group. Hence $P_iP_j=P_jP_i$ for all $i,j$. If $x\in P_i$ and $y\in P_j$, then there exist unique elements $z\in P_j$ and $t\in P_i$ such that $xy=zt$. Since in a left brace $xy=\lambda_x(y)\lambda^{-1}_{\lambda_x(y)}(x)$ and $\lambda_x(y)\in P_j$ and $\lambda^{-1}_{\lambda_x(y)}(x)\in P_i$, we have that $\lambda_x(y)=z$ and $\lambda^{-1}_{\lambda_x(y)}(x)=t$. In particular, if $P_j$ is normal in the multiplicative group of $G$, then
	for every $x\in P_i$ and $y\in P_j$, $xy=xyx^{-1}x$, where $xyx^{-1}\in P_j$, and thus $\lambda_x(y)=xyx^{-1}$.

On the other hand, suppose that $P$ is a finite $p$-group, for a prime $p$, such that $\Aut(P)$ is not a $p$-group. 
Let $q$ be a prime divisor of $|\Aut(P)|$, $q\neq p$. Let $\alpha\in\Aut(P)$ be of order $q$. Suppose that $P$ is an IYB group. 
Consider the semidirect product $P\rtimes\Z/(q)$, where
$$(a_1,b_1)\cdot (a_2,b_2)=(a_1\alpha^{b_1}(a_2),b_1+b_2).$$
Suppose that the answer to Question \ref{general} is affirmative. Since the abelian group $\Z/(q)$ is an IYB group, we have that $P\rtimes\Z/(q)$ also is an IYB group. Let $P_1=P\times \{ 0\}$ and $P_2=\{ 1\}\times \Z/(q)$. There exists a structure of left brace $B$ on $P\rtimes \Z/(q)$ such that $P_1$ is the Sylow $p$-subgroup of the additive group of $B$ (because $P_1$ is normal in the multiplicative group of $B$). Furthermore, there exists $(a,b)\in P\rtimes \Z/(q)$ such that $(a,b)P_2(a,b)^{-1}$ is the Sylow $q$-subgroup of the additive group of $B$. Let $f\colon B\longrightarrow P\rtimes \Z/(q)$ be the map defined by $f(a_1,b_1)=(a,b)^{-1}(a_1,b_1)(a,b)$. Then $f$ is an isomorphism from the multiplicative group of $B$ to the group $P\rtimes \Z/(q)$. We define a structure of left brace $B_1$ on the group $P\rtimes \Z/(q)$ defining a sum $+_1$ on the group $P\rtimes \Z/(q)$ as follows:
$$(a_1,b_1)+_1(a_2,b_2)=f(f^{-1}(a_1,b_1)+f^{-1}(a_2,b_2)),$$
for all $(a_1,b_1),(a_2,b_2)\in P\rtimes \Z/(q)$, where $+$ is the sum of the left brace $B$. Note that $f$ is an isomorphism of left braces from the left brace $B$ to the left brace $B_1$. In particular, $f(P_1)=P_1$ is the Sylow $p$-subgroup of $(B_1,+_1)$ and $P_2=f((a,b)P_2(a,b)^{-1})$ is the Sylow $q$-subgroup of $(B_1,+_1)$. 
By the argument given above, $\lambda_{x}(y)=xyx^{-1}$ for $x\in P_2, y\in P_1$, in the left brace $B_1$, so that we obtain
\begin{eqnarray*}
	\lambda_{(1,1)}(a_1,0)&=&(1,1)(a_1,0)(1,-1)\\
	&=&(\alpha(a_1),1)(1,-1)=(\alpha(a_1),0).
\end{eqnarray*}  
Let $\beta\in \Aut(P_1)$ be the automorphism defined by $\beta(x,0)=(\alpha(x),0)$. Then  $\beta\in \Aut(P_1,+,\cdot)$  because $\lambda_{(1,1)}\in \Aut (B_1,+)$. In particular, this proves that there exists a structure of left brace on the group $P$ such that $\alpha$ is an automorphism of this left brace.

The above observation leads to an interesting consequence. Namely, assume that $B$ is an IYB group of cardinality $p^n$, for a prime $p$. If there exists some $\sigma \in \Aut (B,\cdot)$ of prime order $q\neq p$ such that $\sigma \notin \Aut (B,+,\cdot)$ for every left brace structure $(B,+,\cdot)$ on $B$, then the answer to Question~\ref{general} is negative.

In this paper, we focus on the problem whether finite solvable groups with all Sylow subgroups of nilpotency class at most $2$ are IYB groups.

\section{Results}
It is shown in \cite[Theorem 3.3]{CJR} that if $G$ is a finite group such that $G=AH$, where $A$ is an abelian normal subgroup and $H$ is an IYB subgroup such that $A\cap H=\{ 1\}$, then $G$ is also an IYB group. This is essentially due to the fact that if we consider the trivial structure of brace on the abelian group $A$, then every automorphism of the group $A$ is also an automorphism of the trivial brace $A$.  This motivates our first observation.

\begin{lemma}\label{aut}
	Let $G$ be a nilpotent group of class $2$ and with derived subgroup $G'$ of odd order. Then there exists a structure of left brace on $G$ such that $\Aut(G)=\Aut(G,+,\cdot)$.
\end{lemma}

\begin{proof}
	Since $G'$ has odd order, every element of $G'$ has a unique square root. We define a sum $+$ in $G$ by the rule
	$$h_1+h_2=h_1h_2[h_2,h_1]^{\frac{1}{2}}$$
    for $h_1,h_2\in G$. By \cite{AW}, $(G,+,\cdot)$ is a two-sided brace. Let $f\in\Aut(G)$. We have that
	\begin{align*}f(h_1+h_2)=&f(h_1h_2[h_2,h_1]^{\frac{1}{2}})\\
		=&(f(h_1)f(h_2))[f(h_2),f(h_1)]^{\frac{1}{2}}\\
		=&f(h_1)+f(h_2),
		\end{align*} 
		for all $h_1,h_2\in G$. Hence $f\in\Aut(G,+,\cdot)$ and thus $\Aut(G)=\Aut(G,+,\cdot)$.
\end{proof} 

As a consequence, we get the following result that generalizes \cite[Theorem 3.3]{CJR}.

\begin{proposition}\label{nil}
	Let $G$ be a finite group with a normal nilpotent subgroup $N$ of nilpotency class at most $2$ and with derived subgroup $N'$ of odd order and a subgroup $H$ such that $G=NH$ and $N\cap H=\{ 1\}$. If $H$ is an IYB group, then $G$ is also an IYB group.
\end{proposition}

\begin{proof}
	Note that $G$ is the inner semidirect product of $N$ by $H$. If $N$ is abelian, then by \cite[Theorem 3.3]{CJR} $G$ is an IYB group.
	
	Let $(H,+,\cdot)$ be a structure of left brace on the group $H$.
	Suppose that $N$ is nilpotent of nilpotency class $2$. By Lemma \ref{aut}, there exists a structure of left brace $(N,+,\cdot)$ on the group $N$ such that $\Aut(N)=\Aut(N,+,\cdot)$. We define the sum $+$ on $G$ by
	$$ah_1+bh_2=(a+b)(h_1+h_2)$$
	for all $a,b\in N$ and $h_1,h_2\in H$. Since $\Aut(N)=\Aut(N,+,\cdot)$, we have that $(G,+,\cdot)$ is a left brace, it is the inner semidirect product of the left brace $(N,+,\cdot)$ by the left brace $(H,+,\cdot)$  (see \cite[Section 3]{Csurvey}). Therefore, the result follows.  
\end{proof}

Recall that a finite group $G$ has the Sylow tower property if there exists a normal series
$$\{ 1\}=G_0\subseteq G_1\subseteq\dots\subseteq G_n=G$$
such that $G_{i+1}/G_{i}$ is isomorphic to a Sylow subgroup of $G$ for all $i=0,\dots ,n-1$.
In this case $G$ is solvable and $G_1$ is a normal Sylow subgroup of $G$.

In \cite[Corollary 4.3]{BDG} it is stated that every finite solvable $A$-group is an IYB
group. However, its proof is only valid with the additional hypothesis that the finite solvable $A$-group has the Sylow tower property. The proof of the following result is an easy generalization of the argument used in the proof of \cite[Corollary 4.3]{BDG}.

\begin{theorem}\label{tower}
	Let $G$ be a finite group of odd order or such that all Sylow $2$-subgroups of $G$ are abelian. If $G$ has the Sylow tower property and all the Sylow subgroups of $G$ have nilpotency class at most $2$, then $G$ is an IYB group.
\end{theorem} 

\begin{proof}
	Since $G$ has the Sylow tower property, there exists a normal series
	$$\{ 1\}=G_0\subseteq G_1\subseteq\dots\subseteq G_n=G$$
	such that $G_{i+1}/G_{i}$ is isomorphic to a Sylow subgroup of $G$ for all $i=0,\dots ,n-1$. Note that $n$ is the number of distinct prime divisors of the order of $G$.
	We shall prove the result by induction on $n$.  For $n=1$, $G=G_1$ is a $p$-group for some prime $p$ and by the hypothesis it has nilpotency class at most $2$. Hence $G$ is an IYB group in this case.
	Suppose $n>1$ and that the result holds for $n-1$. Let $p_1,\dots, p_n$ be the distinct prime divisors of the order of $G$, such that $G_{i+1}/G_{i}$ is a $p_{i+1}$-group for all $i=0,\dots ,n-1$. By Schur-Zassenhaus theorem, $G_1$ has a complement $H$ in $G$. Thus $G=G_1H$ is the inner semidirect product of $G_1$ by $H$. It is easy to see that $H\cong G/G_1$ has the Sylow tower property and the number of prime divisors of the order of $H$ is $n-1$. Since the Sylow subgroups of $H$ are also Sylow subgroups of $G$, by the induction hypothesis $H$ is an IYB group. Since $G_1$ has nilpotency class at most $2$, by Proposition \ref{nil}, $G$ is an IYB group. Therefore, the result follows by induction.	 
\end{proof}

A natural question that arises is: what happens if $G$ has a non-abelian $2$-subgroup and $G$ has the Sylow tower property?

In \cite[Theorem 2.1]{CJO20}, to prove that finite solvable $A$-groups are IYB groups, the special structure of $A$-groups decribed in \cite{T} is used. Thus, in order to deal with groups without Sylow tower property, we first describe a special structure of a finite solvable group such that all Sylow subgroups have nilpotency class at most $2$. This will be the key to prove that finite groups of odd order such that all Sylow subgroups have nilpotency class at most $2$ are IYB groups.  Recall that the Fitting subgroup $F(G)$ of a group $G$ is the is the maximal normal nilpotent subgroup of $G$, \cite{Robinson}.	

\begin{theorem}\label{main1}
	Let $G$ be a finite solvable group such that all Sylow subgroups have nilpotency class at most $2$. Then there exist nilpotent subgroups $N_1,\dots , N_k$ and subgroups $G=M_0\supseteq M_1\supseteq\dots\supseteq M_{k-1}\supseteq M_k=\{ 1\}$ such that
	\begin{itemize}
	\item[(i)] $G=N_1\cdots N_iM_i$, for all $1\leq i\leq k$,
	\item[(ii)] $M_{i-1}=N_iM_i$, for all $1\leq i\leq k$,
	\item[(iii)] $N_i$ is normal in $M_{i-1}$, for all $1\leq i\leq k$,
	\item[(iv)] $((\dots ((N_1\cap M_1)N_2)\cap \dots\cap M_{i-1})N_i)\cap M_i$ is normal in $M_i$ and a central subgroup of $F(M_{i-1})F(M_i)$, for all $1\leq i\leq k$. 
		\end{itemize} 
\end{theorem}
\begin{proof}
	 Let $p_1,\dots ,p_m$ be the distinct prime divisors of $|G|$. In this proof, for every subgroup $S$ of $G$, a Sylow $p_r$-subgroup of $S$ will mean a $p_r$-subgroup of $S$ of cardinality $p_r^t$, where $|S|=p_r^tk$ and $p_r$ is not a divisor of $k$. Note that if $p_r$ is not a divisor of $|S|$, then $\{ 1\}$ is a Sylow $p_r$-subgroup of $S$.
	 
	If $G$ is nilpotent, then  we may put $k=1$, $N_1=G$ and $M_1=\{ 1\}$ and the result follows in this case.
	
	Suppose that $G$ is not nilpotent. Let $\pi_1\colon G\longrightarrow G/F(G)$ be the natural map. Let $H_1=\pi_1^{-1}(\Ze(F(G/F(G))))$. Note that $H_1$ is a normal non-nilpotent subgroup of $G$, $F(G)$ is a proper subgroup of $H_1$ such that $H_1/F(G)$ is abelian. Let $N_1=H'_1$ be the derived subgroup of $H_1$.	Then $N_1\subseteq F(G)$.
	Let $P_{1,1},\dots, P_{1,m}$ be a Sylow system of $H_1$, that is $H_1=P_{1,1}\cdots P_{1,m}$, $P_{1,r}$ is a Sylow $p_r$-subgroup of $H_1$ and $P_{1,j}P_{1,r}=P_{1,r}P_{1,j}$, for all $j,r$. Then $M_1=M_G(H_1)=N_G(P_{1,1})\cap\dots \cap N_G(P_{1,m})$ is a system normalizer of $H_1$ relative to $G$.  Thus by a result of Hall, $G=M_1N_1$, see \cite[page 24]{T}. Since $H_1$ is not nilpotent, $M_1\neq G$. 
	
	Since $P_{1,r}$ has nilpotency class at most $2$, the derived subgroup $P'_{1,r}$ of $P_{1,r}$ is central in $P_{1,r}$. Furthermore, $P'_{1,r}\subseteq N_1\subseteq F(G)$. 
	Moreover, since $F(G)$ is nilpotent and normal in $G$, $F(G)$ is a direct product of its Sylow subgroups, so that $P'_{1,r}$ is central in $F(G)$. Since $H_1/F(G)$ is abelian and $H_1$ is normal in $G$, we have that $P_{1,r}F(G)$ is normal in $G$. We shall prove that $P'_{1,r}$ is normal in $G$. Let $g\in G$. Since $P_{1,r}$ is a Sylow subgroup of $P_{1,r}F(G)$, there exists $h\in F(G)$ such that $gP_{1,r}g^{-1}=hP_{1,r}h^{-1}$. Hence $gP'_{1,r}g^{-1}=hP'_{1,r}h^{-1}=P'_{1,r}$, because $P'_{1,r}$ is central in $F(G)$. Hence $P'_{1,r}$ is normal in $G$.
	
	Let $B_1=P'_{1,1}\cdots P'_{1,m}$.  Then $B_1\subseteq N_1$ and $B_1$ is normal in $G$ and central in $F(G)$, because so is every $P_{1,r}'$. Let $\bar G=G/B_1$, $\bar H_1=H_1/B_1$, $\overline{F(G)}=F(G)/B_1$ and $\bar N_1=N_1/B_1$. Note that $\bar N_1$ is the derived subgroup of $\bar H_1$. Since $\bar H_1$ is an $A$-group and a normal subgroup of $\bar G$, by \cite[(4.4)]{T}, $\bar N_1$ is a complement in $\bar G$ to any system normalizer of $\bar H_1$ relative to $\bar G$. In particular $(M_1B_1)\cap N_1=B_1$. Hence, $N_1\cap M_1\subseteq B_1$.
	
	Let $N$ be a nilpotent subgroup of $M_1$. Let $b\in B_1\cap N$. There exist $x_1\in P'_{1,1},\dots ,x_m\in P'_{1,m}$ such that $b=x_1\cdots x_m$. Note that every $x_r$ is a power of $b$. Thus $x_r\in P'_{1,r}\cap N$.
	Since $N\subseteq M_1\subseteq N_G(P_{1,r})$, we have that $P_{1,r}N$ is a subgroup of $G$ and $P_{1,r}$ is normal in $P_{1,r}N$. If $P$ is the Sylow $p_r$-subgroup of $N$, then $P_{1,r}P$ is a Sylow $p_r$-subgroup of $P_{1,r}N$. Since $P_{1,r}P$ has nilpotency class at most $2$, $P'_{1,r}$ is central in $P_{1,r}P$. Hence $P'_{1,r}\subseteq C_G(P)$. Therefore $P'_{1,r}\cap N\subseteq \Ze(P)\subseteq \Ze(N)$.     
	 Hence $x_r\in P'_{1,r}\cap N$ is central in $N$. Thus $b$ is central in $N$.  It follows that $B_1\cap N\subseteq \Ze (N)$. In particular, $B_1\cap F(M_1)\subseteq \Ze (F(M_1))$. Since $N_1\cap M_1$ is a normal nilpotent subgroup of $M_1$ contained in $B_1$, we have  that 
	 $$N_1\cap M_1=B_1\cap M_1=B_1\cap F(M_1)\subseteq \Ze(F(M_1)).$$
    Since $B_1\subseteq \Ze(F(G))$, we have that $N_1\cap M_1\subseteq \Ze(F(G)F(M_1))$,  as desired.
	
	Next, suppose that $i\geq 1$ and we have constructed subgroups $G=M_0\supseteq M_1\supseteq\dots \supseteq M_i$, $M_l\neq M_{l+1}$ for all $0\leq l<i$ and nilpotent subgroups $N_1,\dots ,N_i$ of $G$ such that $M_1,\dots ,M_{i-1}$ are not nilpotent, $N_j$ is the derived subgroup of $H_j=\pi_j^{-1}(\Ze(F(M_{j-1}/F(M_{j-1}))))$, where $\pi_j\colon M_{j-1}\longrightarrow M_{j-1}/F(M_{j-1})$ is the natural map, $M_j=M_{M_{j-1}}(H_j)=N_{M_{j-1}}(P_{j,1})\cap\dots \cap N_{M_{j-1}}(P_{j,m})$ 
	is a system normalizer of $H_j$ relative to $M_{j-1}$, where $P_{j,r}$ is a Sylow $p_r$-subgroup of $H_j$ and $P_{j,1},\dots ,P_{j,m}$ is a Sylow system of $H_j$, for all $1\leq j\leq i$, $B_j=P'_{j,1}\cdots P'_{j,m}$  and
	\begin{itemize}
		\item[(i)] $G=N_1\cdots N_lM_l$, for all $1\leq l\leq i$,
		\item[(ii)] $M_{l-1}=N_lM_l$, for all $1\leq l\leq i$,
		\item[(iii)] $N_l$ is normal in $M_{l-1}$, for all $1\leq l\leq i$,
		\item[(iv)] $((\dots ((N_1\cap M_1)N_2)\cap \dots\cap M_{l-1})N_l)\cap M_l$ is normal in $M_l$ and a central subgroup of $F(M_{l-1})F(M_l)$, for all $1\leq l\leq i$.
	\end{itemize} 
    We also assume that
	$$((\dots ((N_1\cap M_1)N_2)\cap \dots\cap M_{l-1})N_l)\cap M_l=((\dots ((B_1\cap M_1)B_2)\cap \dots\cap M_{l-1})B_l)\cap M_l$$
	and that the Sylow $p_r$-subgroup of 
	$$((\dots ((B_1\cap M_1)B_2)\cap \dots\cap M_{l-1})B_l)\cap M_l$$
	is contained in the $p_r$-subgroup $P'_{1,r}P'_{2,r}\cdots P'_{l,r}$ of $G$, for all $1\leq l\leq i$ and $1\leq r\leq m$. Note that this holds for $i=1$.
	 
	If $M_i$ is nilpotent, then we define $k=i+1$, $N_{i+1}=M_i$ and $M_{i+1}=\{ 1\}$ and the result follows in this case.
	
	Suppose that $M_i$ is not nilpotent.  
	
	Let $\pi_{i+1}\colon M_i\longrightarrow M_i/F(M_i)$ be the natural map. Let $$H_{i+1}=\pi_{i+1}^{-1}(\Ze(F(M_i/F(M_i)))).$$ Note that $H_{i+1}$ is a normal subgroup of $M_i$, $F(M_i)$ is a proper subgroup of $H_{i+1}$ such that $H_{i+1}/F(M_i)$ is abelian.  Let $N_{i+1}=H'_{i+1}$ be the derived subgroup of $H_{i+1}$.	Then $N_{i+1}\subseteq F(M_i)$.
	Let $P_{i+1,1},\dots, P_{i+1,m}$ be a Sylow system of $H_{i+1}$ such that $P_{i+1,j}$ is a Sylow $p_j$-subgroup of $H_{i+1}$. Then $M_{i+1}=M_{M_i}(H_{i+1})=N_{M_i}(P_{i+1,1})\cap\dots \cap N_{M_i}(P_{i+1,m})$ is a system normalizer of $H_{i+1}$ relative to $M_i$.  Then by a result of Hall, $M_i=M_{i+1}N_{i+1}$, see \cite[page 24]{T}. Hence $N_{i+1}$ is a normal subgroup of $M_i$ and
	 $$G=N_1\cdots N_iM_i=N_1\cdots N_iM_{i+1}N_{i+1}=N_1\cdots N_{i+1}M_{i+1}.$$
	
	Since $F(M_i)$ is the maximal nilpotent normal subgroup of $M_i$, we have that $H_{i+1}$ is not nilpotent. Hence $M_{i+1}\neq M_i$  (as otherwise all $P_{i+1,j}$ are normal subgroups in $M_i$, hence in $H_{i+1}$, and $H_{i+1}$ would be nilpotent).

	Let $B_{i+1}=P'_{i+1,1}\cdots P'_{i+1,m}$. Similarly as we have proved for $B_1$, one shows that $B_{i+1}\subseteq N_{i+1}$, $B_{i+1}$ is normal in $M_i$ and $B_{i+1}$ is a central subgroup of $F(M_i)$.

	Since,  by (iv), 
	$$((\dots ((N_1\cap M_1)N_2)\cap \dots\cap M_{i-1})N_i)\cap M_i$$
	is a normal subgroup of $M_i$ contained in $\Ze(F(M_i))$, we have that 
	$$(((\dots ((N_1\cap M_1)N_2)\cap \dots\cap M_{i-1})N_i)\cap M_i)B_{i+1}$$
	is normal in $M_i$ and a central subgroup of $F(M_i)$. Since
	$$((\dots ((N_1\cap M_1)N_2)\cap \dots\cap M_{i-1})N_i)\cap M_i=((\dots ((B_1\cap M_1)B_2)\cap \dots\cap M_{i-1})B_i)\cap M_i$$
	 we have that
		\begin{eqnarray*}\lefteqn{(((\dots ((N_1\cap M_1)N_2)\cap \dots\cap M_{i-1})N_i)\cap M_i)B_{i+1}}\\
			&=&(((\dots ((B_1\cap M_1)B_2)\cap \dots\cap M_{i-1})B_i)\cap M_i)B_{i+1}\end{eqnarray*}
			is a central subgroup of $F(M_i)$.
			Let
			$$B=(((\dots ((B_1\cap M_1)B_2)\cap \dots\cap M_{i-1})B_i)\cap M_i)B_{i+1}.$$
    Since, by our inductive hypothesis, the Sylow $p_r$-subgroup of
		$$((\dots ((B_1\cap M_1)B_2)\cap \dots\cap M_{i-1})B_i)\cap M_i$$
		is contained in the $p_r$-subgroup $P'_{1,r}P'_{2,r}\cdots P'_{i,r}$ of $G$, for all $1\leq r\leq m$, 
	and since $B$ 
	is abelian, the Sylow $p_r$-subgroup of $B$ is contained in $P'_{1,r}P'_{2,r}\cdots P'_{i,r}P'_{i+1,r}$, which is a $p_r$-subgroup of $G$ because $P'_{i+1,r}\subseteq M_l\subseteq N_{G}(P_{l,r})$, for all $1\leq l\leq i$.	
     Let $\overline M_i=M_i/B$, $$\overline H_{i+1}=((((\dots ((B_1\cap M_1)B_2)\cap \dots\cap M_{i-1})B_i)\cap M_i)H_{i+1})/B,$$ $\overline{F(M_i)}=F(M_i)/B$ and $$\overline N_{i+1}=((((\dots ((B_1\cap M_1)B_2)\cap \dots\cap M_{i-1})B_i)\cap M_i)N_{i+1})/B.$$ 
	Note that $\overline N_{i+1}$ is the derived subgroup of $\overline H_{i+1}$. Since $\overline H_{i+1}$ is an $A$-group and a normal subgroup of $\overline M_i$, by \cite[(4.4)]{T}, $\overline N_{i+1}$ is a complement in $\overline M_i$ to any system normalizer of $\overline H_{i+1}$ relative to $\overline M_i$. In particular $$(M_{i+1}B)\cap ((((\dots ((B_1\cap M_1)B_2)\cap \dots\cap M_{i-1})B_i)\cap M_i)N_{i+1})=B.$$ 
	Hence, $$((((\dots ((B_1\cap M_1)B_2)\cap \dots\cap M_{i-1})B_i)\cap M_i)N_{i+1})\cap M_{i+1}\subseteq B.$$
	Hence
	\begin{eqnarray*}\lefteqn{((\dots ((N_1\cap M_1)N_2)\cap \dots\cap M_{i})N_{i+1})\cap M_{i+1}}\\
		&=&B\cap M_{i+1}=((\dots ((B_1\cap M_1)B_2)\cap \dots\cap M_{i})B_{i+1})\cap M_{i+1},\end{eqnarray*} 
	and its Sylow $p_r$-subgroup is contained in $P'_{1,r}P'_{2,r}\cdots P'_{i+1,r}$.
	
	Let $N$ be a nilpotent subgroup of $M_{i+1}$. Let $b\in N\cap B$. There exist $x_{j,1}\in P'_{j,1},\dots ,x_{j,m}\in P'_{j,m}$, for $1\leq j\leq i+1$ such that 
	$$b=x_{1,1}\cdots x_{i+1,1}\cdots x_{1,m}\cdots x_{i+1,m}.$$ 
	Since $B$ is abelian, $x_{1,r}x_{2,r}\cdots x_{i+1,r}$ is a power of $b$ and thus $x_{1,r}x_{2,r}\cdots x_{i+1,r}\in N\cap B$. We shall prove that $x_{1,r}x_{2,r}\cdots x_{i+1,r}\in \Ze(N)$.  Since $N\subseteq N_G(P_{1,r})\cap\dots \cap N_G(P_{i+1,r})$ and $P_{j+1,r}\subseteq N_G(P_{j,r})$, we have that $P_{1,r}\cdots P_{i+1,r}N$ is a subgroup of $G$ and $P_{1,r}\cdots P_{i+1,r}$ is normal in $P_{1,r}\cdots P_{i+1,r}N$. If $Q$ is the Sylow $p_r$-subgroup of $N$, then $P_{1,r}\cdots P_{i+1,r}Q$ is a Sylow $p_r$-subgroup of $P_{1,r}\cdots P_{i+1,r}N$. Since $P_{1,r}\cdots P_{i+1,r}Q$ has nilpotency class at most $2$, $P'_{1,r}\cdots P'_{i+1,r}$ is central in $P_{1,r}\cdots P_{i+1,r}Q$. Hence $P'_{1,r}\cdots P'_{i+1,r}\subseteq C_G(Q)$. Therefore 
	$$(P'_{1,r}\cdots P'_{i+1,r})\cap N\subseteq \Ze(Q)\subseteq \Ze(N).$$     
	Hence $x_{1,r}\cdots x_{i+1,r}\in (P'_{1,r}\cdots P'_{i+1,r})\cap N$ is central in $N$. Thus $b$ is central in $N$. So it follows that $N\cap B\subseteq \Ze (N)$. Since $((((\dots ((N_1\cap M_1)N_2)\cap \dots\cap M_{i-1})N_i)\cap M_i)N_{i+1})\cap M_{i+1}=B\cap M_{i+1}$ is a normal nilpotent subgroup of $M_{i+1}$ contained in $B$, this implies that
	$$((((\dots ((N_1\cap M_1)N_2)\cap \dots\cap M_{i-1})N_i)\cap M_i)N_{i+1})\cap M_{i+1}\subseteq \Ze(F(M_{i+1})).$$
	Since $B$ is central in $F(M_i)$, we get that
	$$((((\dots ((N_1\cap M_1)N_2)\cap \dots\cap M_{i-1})N_i)\cap M_i)N_{i+1})\cap M_{i+1}\subseteq \Ze(F(M_i)F(M_{i+1})).$$
	Hence, the inductive step is proved. Since the inductive procedure leads to a strictly decreasing chain of subgroups $G=M_0\supset \dots \supset M_i$ if $M_l$ is not nilpotent for $0\leq l<i$, it has to terminate at the trivial subgroup. Therefore the result follows.	  
\end{proof}

Now we are ready to generalize Theorem \ref{tower} without the hypothesis on the Sylow tower property. We note that the proofs of the following results are based on a much more complicated construction.

\begin{theorem}\label{main}
	Let $G$ be a finite group of odd order such that all Sylow subgroups have nilpotency class at most $2$. Then $G$ is an IYB group. 
\end{theorem}
 
 \begin{proof} By Feit--Thompson theorem, $G$ is solvable. By Theorem \ref{main1}, there exist nilpotent subgroups $N_1,\dots , N_k$ and subgroups $G=M_0\supseteq M_1\supseteq\dots\supseteq M_{k-1}\supseteq M_k=\{ 1\}$ such that
 	\begin{itemize}
 		\item[(i)] $G=N_1\cdots N_iM_i$, for all $1\leq i\leq k$,
 		\item[(ii)] $M_{i-1}=N_iM_i$, for all $1\leq i\leq k$,
 		\item[(iii)] $N_i$ is normal in $M_{i-1}$, for all $1\leq i\leq k$,
 		\item[(iv)] $((\dots ((N_1\cap M_1)N_2)\cap \dots\cap M_{i-1})N_i)\cap M_i$ is normal in $M_i$ and a central subgroup of $F(M_{i-1})F(M_i)$, for all $1\leq i\leq k$.	
 	\end{itemize} 
 	In particular $G=N_1\cdots N_k$  and $M_i = N_{i+1}\cdots N_{k}=  N_k\cdots N_{i+1}$ for $i< k-1$. We define a sum on each $N_i$ by the rule
 	\begin{equation} \label{N1} x+y=xy[y,x]^{\frac{1}{2}},
    \end{equation}
 	for all $x,y\in N_i$. Note that $(N_i,+,\cdot)$ is a left brace by \cite{AW}.  Thus, we may assume that $k>1$. We define a sum on $G$ by the rule
 	\begin{equation} \label{Nk} (x_1\cdots x_k)+(y_1\cdots y_k)=(x_1+y_1)\cdots (x_k+y_k),
    \end{equation} 
 	for all $x_1,y_1\in N_1,\,\dots ,\, x_k,y_k\in N_k$. First, we shall prove that this sum is well-defined.
 	So assume that $x_1,y_1,z_1,t_1\in N_1,\,\dots ,\, x_k,y_k,z_k,t_k\in N_k$ are elements such that
 	$$x_1\cdots x_k=z_1\cdots z_k\mbox{ and }y_1\cdots y_k=t_1\cdots t_k.$$
 	Note that,  by (iv),
 	$$
    z_1^{-1}x_1=z_2\cdots z_kx_k^{-1}\cdots x_2^{-1}\in N_1\cap M_1\subseteq \Ze(F(G)F(M_1)),$$
     and $z_2 \in N_2\subseteq F(M_1)$, so that $z_{1}^{-1}x_1$ commutes with $z_2$. Hence
 	\begin{eqnarray*}\lefteqn{z_1^{-1}x_1z_2^{-1}x_2}\\
 		&=&z_2^{-1}z_1^{-1}x_1x_2\\
 		&=&z_3\cdots z_kx_k^{-1}\cdots x_3^{-1}\in ((N_1\cap M_1)N_2)\cap M_2\subseteq\Ze(F(M_1)F(M_2)).
 		\end{eqnarray*} 
 	Then, using induction it is easy to see that $z_{i-1}^{-1}\cdots z_1^{-1}x_1\cdots x_{i-1}$ commutes with  $z_i$ and
 	\begin{eqnarray*}\lefteqn{z_1^{-1}x_1z_2^{-1}x_2\cdots z_i^{-1}x_i}\\
 		&=&z_i^{-1}\cdots z_1^{-1}x_1\cdots x_i\\
 		&=&z_{i+1}\cdots z_kx_k^{-1}\cdots x_{i+1}^{-1}\\
        &\in& ((\dots ((N_1\cap M_1)N_2)\cap \dots \cap M_{i-1})N_i)\cap M_i\subseteq\Ze(F(M_{i-1})F(M_i)),
 		\end{eqnarray*} 
 	for all $1\leq i\leq k$. In particular,
 	\begin{eqnarray} \label{CGNN} 
    \lefteqn{z_1^{-1}x_1z_2^{-1}x_2\cdots z_i^{-1}x_i} \nonumber\\
    &=&z_i^{-1}\cdots z_1^{-1}x_1\cdots x_i=z_{i+1}\cdots z_kx_k^{-1}\cdots x_{i+1}^{-1}\in C_G(N_{i}N_{i+1}),
       \end{eqnarray}
 	for all $1\leq i<k$. Similarly we get
 	 	\begin{eqnarray} \label{CGNN2} 
        \lefteqn{t_1^{-1}y_1t_2^{-1}y_2\cdots t_i^{-1}y_i} \nonumber\\
          &=&t_i^{-1}\cdots t_1^{-1}y_1\cdots y_i=t_{i+1}\cdots t_ky_k^{-1}\cdots y_{i+1}^{-1}\in C_G(N_{i}N_{i+1}),
        \end{eqnarray}
 	for all $1\leq i<k$.
 	We have that
 	\begin{eqnarray*}
 		\lefteqn{(x_1+y_1)\cdots (x_k+y_k)}\\
 		&=&x_1y_1[y_1,x_1]^{\frac{1}{2}}x_2y_2[y_2,x_2]^{\frac{1}{2}}\cdots x_ky_k[y_k,x_k]^{\frac{1}{2}}\\
 		&=&z_1(z_2\cdots z_kx_k^{-1}\cdots x_2^{-1})t_1(t_2\cdots t_ky_k^{-1}\cdots y_2^{-1})[y_1,x_1]^{\frac{1}{2}}x_2y_2[y_2,x_2]^{\frac{1}{2}}\\
 		&&\cdots x_ky_k[y_k,x_k]^{\frac{1}{2}}\\
 		&=& z_1t_1[y_1,x_1]^{\frac{1}{2}}(z_2\cdots z_kx_k^{-1}\cdots x_2^{-1})x_2(t_2\cdots t_ky_k^{-1}\cdots y_2^{-1})y_2[y_2,x_2]^{\frac{1}{2}}\\
 		&&\cdots x_ky_k[y_k,x_k]^{\frac{1}{2}}\\
 		&=&\cdots =z_1t_1[y_1,x_1]^{\frac{1}{2}}z_2t_2[y_2,x_2]^{\frac{1}{2}}\cdots z_kt_k[y_k,x_k]^{\frac{1}{2}}.
 	\end{eqnarray*}
    In the third equality above we use the fact that the elements $t_1, [y_1,x_{1}]^{-1} \in N_1$ and $x_2\in N_2$ commute with $z_2\cdots z_kx_k^{-1}\cdots x_2^{-1}$ and $t_2\cdots t_ky_k^{-1}\cdots y_2^{-1}$, which is a consequence of (\ref{CGNN}) and (\ref{CGNN2}) for $i=1$. Then, in the next equalities one applies  (\ref{CGNN}) and (\ref{CGNN2}) for the subsequent values of $i$.
 	But (\ref{CGNN}) implies also that
 	\begin{eqnarray*}
 	    \lefteqn{z_i^{-1}x_i
        =(x_{i-1}^{-1}z_{i-1}\cdots x_1^{-1}z_1)(z_{i+1}\cdots z_kx_k^{-1}\cdots x_{i+1}^{-1})}\\
        &\in& (C_G(N_{i-1}N_i)C_G(N_iN_{i+1}))\cap N_i\subseteq\Ze(N_i),
        \end{eqnarray*}
 	and similarly $t_i^{-1}y_i\in\Ze(N_i)$. Hence $[y_i,x_i]=[t_i,z_i]$, for all $1\leq i\leq k$. Therefore
 	$$(x_1+y_1)\cdots (x_k+y_k)=(z_1+t_1)\cdots (z_k+t_k)$$
 	and the sum on $G$ is well-defined.
 	
Note that the restriction of the sum in $G$ to $N_i$ is exactly the sum on $N_i$ defined previously. Since each $(N_i,+,\cdot)$ is a left brace, it is easy to see that $(G,+)$ is an abelian group,  in particular $-x_1\cdots x_n = (-x_1)\cdots (-x_n)$.
Let $a_1,b_1,c_1\in N_1\,\dots ,\, a_k,b_k,c_k\in N_k$. We have that
 	\begin{eqnarray}\label{brace} 
 		\lefteqn{a_1\cdots a_k(b_1\cdots b_k+c_1\cdots c_k)}\\
 		&=&a_1\cdots a_k(b_1+c_1)\cdots (b_k+c_k)\notag\\
 		&=&a_1\cdots a_kb_1c_1[c_1,b_1]^{\frac{1}{2}}\cdots b_kc_k[c_k,b_k]^{\frac{1}{2}}\notag\\
 		&=&a_1(a_2\cdots a_kb_1c_1[c_1,b_1]^{\frac{1}{2}}(a_2\cdots a_k)^{-1})\notag \\
 		&&\cdot a_2\cdots a_kb_2c_2[c_2,b_2]^{\frac{1}{2}}\cdots b_kc_k[c_k,b_k]^{\frac{1}{2}}\notag \\
 		&=&a_1(a_2\cdots a_kb_1c_1[c_1,b_1]^{\frac{1}{2}}(a_2\cdots a_k)^{-1})\notag \\
 		&&\cdot a_2(a_3\cdots a_k)b_2c_2[c_2,b_2]^{\frac{1}{2}}(a_3\cdots a_k)^{-1}\cdots a_kb_kc_k[c_k,b_k]^{\frac{1}{2}}\notag \\
 		&=&a_1(a_2\cdots a_kb_1(a_2\cdots a_k)^{-1}+a_2\cdots a_kc_1(a_2\cdots a_k)^{-1})\notag \\
 		&&\cdot a_2(a_3\cdots a_kb_2(a_3\cdots a_k)^{-1}+a_3\cdots a_kc_2(a_3\cdots a_k)^{-1})\cdots a_k(b_k+c_k)\notag \\
 		&=&(a_1a_2\cdots a_kb_1(a_2\cdots a_k)^{-1}+a_1a_2\cdots a_kc_1(a_2\cdots a_k)^{-1}-a_1)\notag \\
 		&&\cdot (a_2a_3\cdots a_kb_2(a_3\cdots a_k)^{-1}+a_2a_3\cdots a_kc_2(a_3\cdots a_k)^{-1}-a_2)\notag \\
 		&&\cdots  (a_kb_k+a_kc_k-a_k)\notag \\
 		&=&a_1a_2\cdots a_kb_1b_2\cdots b_k+a_1a_2\cdots a_kc_1c_2\cdots c_k-a_1a_2\cdots a_k\notag 		
 		\end{eqnarray} 
        where the last three equalities follow by (\ref{N1}), because $N_i$ are left braces, and by (\ref{Nk})  and Lemma~\ref{aut} (as conjugation in $N_{i}\cdots N_{k}$ induces an automorphism on the normal subgroup $N_i$), respectively.
 	   Hence $(G,+,\cdot)$ is a left brace and the result follows.  
 	   \end{proof}

Consider the group $G=(\Z/(7))\rtimes (\Z/(3))$, where
$$(a_1,a_2)(b_1,b_2)=(a_1+2^{a_2}b_1,a_2+b_2),$$
for all $a_1,b_1\in\Z/(7)$ and all $a_2,b_2\in\Z/(3)$. In this group, using the notation of Theorem \ref{main1}, we have that $F(G)=(\Z/(7))\times \{ 0\}$, $H_1=G$, $N_1=F(G)$. We take the following Sylow system of $H_1$:
$$P_{1,1}=F(G),\quad P_{1,2}=\{ 0\}\times \Z/(3).$$
Then $M_1=N_G(P_{1,1})\cap N_G(P_{1,2})=G\cap P_{1,2}=P_{1,2}$. Hence $N_2=M_1$ and $M_2=\{ (0,0)\}$. Note that $N_1$ and $N_2$ are abelian in this case. Now, the sum on $G$, defined as in Theorem \ref{main1}, is
$$(a_1,a_2)+(b_1,b_2)=(a_1+b_1,a_2+b_2),$$
for all $a_1,b_1\in \Z/(7)$ and all $a_2,b_2\in\Z/(3)$. Thus, the structure of left brace on $G$ is the semidirect product (see \cite[Section 3]{Csurvey}) of the trivial braces $Z/(7)$  by $Z/(3)$. Now we have
$$((0,1)+(0,1))(1,0)+(1,0)=(0,2)(1,0)+(1,0)=(5,2)$$
and
$$(0,1)(1,0)+(0,1)(1,0)=(2,1)+(2,1)=(4,2)\neq (5,2).$$
Hence $G$ is not a two-sided brace, in contrast to the situation described in Lemma~\ref{aut} and applied in the subsequent proofs.   

\begin{theorem}\label{maineven}
	Let $G$ be a finite solvable group of even order such that all Sylow subgroups have nilpotency class at most $2$ and all Sylow $2$-subgroups are abelian. Then $G$ is an IYB group. 
\end{theorem}

\begin{proof}
	The proof is the same as the proof of Theorem \ref{main}. Note that the restriction of the sum on the Sylow $2$-subgroup of each $N_i$ is defined by
	$$a+b=ab[a,b]^{\frac{1}{2}}=ab$$
	because all Sylow $2$-subgroups of $G$ are abelian by hypothesis.
\end{proof}

 We conclude with a special case of Question~\ref{general} that seems natural in the context of our main results.
\begin{question}\label{Q1}
	Let $G$ be a finite solvable group such that all Sylow subgroups have nilpotency class at most $2$. Is $G$ an IYB group?
\end{question}
 
 Theorem \ref{main} shows that if the order of $G$ is odd, then the answer is affirmative. Theorem \ref{maineven} gives an affirmative answer if the Sylow $2$-subgroups of $G$ are abelian.
Let $G$ be a finite solvable group such that all Sylow subgroups have nilpotency class at most $2$. If a Sylow $2$-subgroup $P$ of $G$ admits a structure of left brace such that $\Aut(P)=\Aut(P,+,\cdot)$ then an approach as in the proof of Theorem~\ref{main} allows to prove that $G$ is an IYB group (this condition is used in the fifth equality of (\ref{brace})). On the other hand, all left braces with multiplicative group isomorphic to the dihedral group of order $8$ are described in \cite{B2015} and a direct verification of these braces shows that in all of them there exists a multiplicative automorphism which is not an additive homomorphism. Therefore, another approach would be needed to solve Question~\ref{Q1} in full generality.

\section{Skew braces of nilpotent type}
Recall that a skew left brace, \cite{GV},  is a set $B$ with two binary operations $+$ and $\circ$, such that $(B,+)$ and $(B,\circ)$ are groups, the additive group and the multiplicative group of $B$ respectively, and for every $a,b,c\in B$,
\begin{equation}\label{skew}
	a\circ (b+c)=a\circ b-a+a\circ c.
	\end{equation} 
If $\aleph$ is a class of groups, we say that the skew left brace $B$ is of type $\aleph$ if the additive group of $B$ is in $\aleph$. In particular, skew left braces of abelian type are left braces.

The following result is due to Byott \cite{byott}.
\begin{theorem}\label{Byott}[Byott] The multiplicative group of every finite skew left brace of nilpotent type is solvable.
	\end{theorem} 
Let $G$ be a nilpotent group. We define on $G$ a sum by the rule
$$a+b=ab,$$
for all $a,b\in G$. Then
$$a(b+c)=abc=aba^{-1}ac=ab-a+ac,$$
for all $a,b,c\in G$. Hence $(G,+,\cdot)$ is a skew left brace of nilpotent type. 

\begin{conjecture}
	Every finite solvable group is isomorphic to the multiplicative group of a skew left brace of nilpotent type.
\end{conjecture} 

Our aim is to prove some results that support this conjecture.

\begin{theorem}\label{towerskew}
	Let $G$ be a finite group with the Sylow tower property. Then $G$ is isomorphic to the multiplicative group of a skew left brace of nilpotent type.
\end{theorem}
\begin{proof}
Let $n$ be the number of prime divisors of $|G|$. We shall prove the result by induction on $n$. For $n=1$, $G$ is nilpotent and the result holds by the above argument in this case.

Suppose that $n>1$ and the result holds for every finite group $H$ with the Sylow tower property such that the number of prime divisors of $|H|$ is less than $n$. Since $G$ has the Sylow tower property, $G$ has a normal Sylow subgroup $P$. By the Schur--Zassenhaus theorem, $P$ has a complement $H$ in $G$, that is, $H$ is a subgroup of $G$ such that $G=PH$ and $P\cap H=\{ 1\}$. Note that $H$ also has the Sylow tower property and the number of prime divisors of $H$ is $n-1$. Hence, by the induction hypothesis, we can define a sum $+$ on $H$ such that $(H,+,\cdot)$ is a skew left brace of nilpotent type. Since $P$ is nilpotent, we define a sum $+$ on $P$ by the rule
$$a+b=ab,$$
for all $a,b\in P$. Thus $(P,+,\cdot)$ is a skew left brace of nilpotent type. Now we define a sum $+$ on $G$ by the rule
$$ac+bd=(a+b)(c+d),$$
for all $a,b\in P$ and all $c,d\in H$. Then it is clear that $(G,+)$ is the inner direct product of $(P,+)$ and $(H,+)$. Hence $(G,+)$ is a nilpotent group. Furthermore, for every $a,b,x\in P$ and $c,d,y\in H$ we have that
\begin{eqnarray*}
ac(bd+xy)
    &=& ac(b+x)(d+y)=acbx(d+y)\\
	&=& acbxc^{-1}c(d+y)=acbxc^{-1}(cd-c+cy)\\
	&=& acbc^{-1}a^{-1}acxc^{-1}(cd-c+cy)\\
	&=& (acbc^{-1}-a+acxc^{-1})(cd-c+cy)\\
	&=& acbc^{-1}cd-ac+acxc^{-1}cy\\
	&=& acbd-ac+acxy.
	\end{eqnarray*} 
Hence, $(G,+,\cdot)$ is a skew left brace of nilpotent type. In fact $(G,+,\cdot)$ is the inner semidirect product of the trivial skew left brace $(P,+,\cdot)$ by the skew left brace $(H,+,\cdot)$. Therefore, the result follows by induction.    	
\end{proof}	

\begin{corollary}
	 Every finite supersolvable group is isomorphic to the multiplicative group of a skew left brace of nilpotent type. 
\end{corollary} 
\begin{proof}
	It is well known that every finite supersolvable group has the Sylow tower property. Hence, the result follows by Theorem \ref{towerskew}.
\end{proof}	

\begin{theorem}\label{mainskew}
	Let $G$ be a finite solvable group such that all Sylow subgroups have nilpotency class at most $2$. Then $G$ is isomorphic to the multiplicative group of  a skew left brace of nilpotent type. 
\end{theorem}

\begin{proof}
	By Theorem \ref{main1}, there exist nilpotent subgroups $N_1,\dots , N_k$ and subgroups $G=M_0\supseteq M_1\supseteq\dots\supseteq M_{k-1}\supseteq M_k=\{ 1\}$ such that
	\begin{itemize}
		\item[(i)] $G=N_1\cdots N_iM_i$, for all $1\leq i\leq k$,
		\item[(ii)] $M_{i-1}=N_iM_i$, for all $1\leq i\leq k$,
		\item[(iii)] $N_i$ is normal in $M_{i-1}$, for all $1\leq i\leq k$,
		\item[(iv)] $((\dots ((N_1\cap M_1)N_2)\cap \dots\cap M_{i-1})N_i)\cap M_i$ is normal in $M_i$ and a central subgroup of $F(M_{i-1})F(M_i)$, for all $1\leq i\leq k$.	
	\end{itemize} 
	In particular $G=N_1\cdots N_k$. We define a sum on each $N_i$ by the rule
	\begin{equation} \label{N1skew} x+y=xy,
	\end{equation}
	for all $x,y\in N_i$. Note that $(N_i,+,\cdot)$ is a skew left brace of nilpotent type.  Thus, we may assume that $k>1$. We define a sum on $G$ by the rule
	\begin{equation} \label{Nkskew} (x_1\cdots x_k)+(y_1\cdots y_k)=(x_1+y_1)\cdots (x_k+y_k),
	\end{equation} 
	for all $x_1,y_1\in N_1,\,\dots ,\, x_k,y_k\in N_k$.
	Similarly as in the proof of Theorem \ref{main}, we prove that this sum on $G$ is well-defined and that $(G,+,\cdot)$ is a skew left brace.
	Note that $f\colon N_1\times \dots\times N_k\longrightarrow (G,+)$ defined by $f(x_1,\dots ,x_k)=x_1\cdots x_k$, for all $x_1\in N_1,\,\dots ,\, x_k\in N_k$, is an epimorphism of groups. Hence $(G,+)$ is a nilpotent group and the result follows.
\end{proof}

{\bf Acknowledgments.} The first author was partially supported by
	the grant MINECO PID2023-147110NB-I00 (Spain).


\end{document}